\documentclass[11pt,reqno]{amsart}

\usepackage{amssymb}

\numberwithin{equation}{section}

\begin{document}

\newtheorem{theorem}{Theorem}[section]
\newtheorem{proposition}[theorem]{Proposition}
\newtheorem{lemma}[theorem]{Lemma}
\newtheorem{corollary}[theorem]{Corollary}

\theoremstyle{definition}
\newtheorem{definition}[theorem]{Definition}
\newtheorem{assumption}[theorem]{Assumption}

\newcommand\Ga{{\Gamma}}
\newcommand\Z{{\mathbb Z}}

\title[Sasakian and parabolic Higgs bundles]{Sasakian and parabolic Higgs bundles}

\author[I. Biswas]{Indranil Biswas}

\address{School of Mathematics, Tata Institute of Fundamental
Research, Homi Bhabha Road, Mumbai 400005, India}

\email{indranil@math.tifr.res.in}

\author[M. Mj]{Mahan Mj}

\address{School of Mathematics, Tata Institute of Fundamental
Research, Homi Bhabha Road, Mumbai 400005, India}

\email{mahan@math.tifr.res.in}



\date{}

\begin{abstract}
Let $M$ be a quasi-regular compact connected Sasakian manifold, and let $N\,=\, M/S^1$ be the base
projective 
variety. We establish an equivalence between the class of Sasakian $G$--Higgs bundles over 
$M$ and the class of parabolic (or equivalently, ramified) $G$--Higgs bundles over the 
base $N$.
\end{abstract}

\maketitle

\section{Introduction}\label{se1} 

Let $M$ be a quasi-regular compact Sasakian manifold. The circle group $S^1\,=\,
\text{U}(1)$ acts on
$M$; let $N$ be the corresponding quotient, which is a normal complex projective variety.
Let $G$ be a complex reductive affine algebraic group.
Sasakian Higgs bundles on $M$ with structure group $G$ can be looked at from a 
number of different points of view:

\begin{enumerate}
\item[(1)] As a holomorphic Sasakian principal $G-$bundle over $M$ equipped with a Higgs 
field.

\item[(2)] As a ramified $G$--Higgs bundle over $N$ with ramification locus a normal crossing 
divisor $D$ in $N$.

\item[(3)] As a parabolic $G$--Higgs bundle over $N$ with parabolic structure over $D$ and 
rational parabolic weights.
\end{enumerate} 

The equivalence of (2) and (3) was established in \cite{Bi2}; it should be clarified
that the constructions, and methods, in \cite{Bi2} were greatly motivated by
\cite{BBN2}, \cite{Bi1}. The purpose of this 
paper is to establish an equivalence between (1) and (3). It should be mentioned that
while the Sasakian manifolds are never complex manifolds (their dimension is odd), there
is a natural procedure of defining holomorphic objects on them; the Reeb vector field plays
a crucial role in this process. The details are recalled in Section \ref{se:holvec}.

Let $\Gamma$ be the fundamental 
group of $M$. In \cite{BM} we established a fourth equivalence, the 
Donaldson-Corlette-Hitchin-Simpson correspondence between representations of $\Gamma$ and 
Sasakian $G-$Higgs bundles: \\ Any homomorphism
$$\rho\, :\, \Gamma\,\longrightarrow\, G$$ with the Zariski closure of $\rho(\Gamma)$
reductive canonically gives a virtually basic polystable principal $G$--Higgs bundle on
$M$ with vanishing rational characteristic classes. Conversely, any 
virtually basic polystable principal $G$--Higgs bundle on $M$ with 
vanishing rational characteristic classes corresponds to a flat principal
$G$--bundle on $M$ with the property that the Zariski closure of the monodromy
representation is reductive \cite{BM}. 

Thus, the equivalence of (1) and (3), proven in this paper, completes the equivalence of 
these four perspectives.

\section{Higgs bundles on quasi-regular Sasakian manifolds}

\subsection{Sasakian manifolds}\label{se2.1}

Let $M$ be a quasi-regular compact Sasakian manifold of dimension $2d+1$ (see \cite{BG} 
for definitions and properties). Then ${\rm U}(1)$ acts on $M$ and the action is free 
over a dense open subset of $M$. The action is free everywhere if and only if $M$ is 
regular. Without loss of generality, we assume that the action of ${\rm U}(1)$ on $M$ is 
on the left. (This is not important as ${\rm U}(1)$ is abelian.) Let $$N\,:=\, {\rm 
U}(1)\backslash M$$ be the quotient space; we will refer to $N$
as the base space for $M$. This base $N$ is a normal complex projective 
variety. The subset of $M$ over which the action of ${\rm U}(1)$ is not free will be 
denoted by $D_M$. Let
\begin{equation}\label{p}
p\,:\, M\,\longrightarrow\, N
\end{equation}
be the quotient map, so that $p$ gives a $C^\infty$ principal ${\rm U}(1)$--bundle over the
complement $N- p(D_M)$. Although the structure group acts on the right for a principal
bundle, since ${\rm U}(1)$ is abelian we do not
need to distinguish between left and rights actions.

\begin{assumption}\label{assum}
We assume the following:
\begin{enumerate}
\item The quotient space $N$ is a smooth variety.

\item The subset $D_N\, :=\, p(D_M)\, \subset\, N$ is a simple normal
crossing divisor.
\end{enumerate}
\end{assumption}

A simple minded example would be the following: Let $B_0$ be an orbifold surface, and let $p_1\, 
:\, S\, \longrightarrow\, B_0$ a principal $\text{U}(1)$ bundle such that the characteristic 
class of it is nonzero. So $S$ is a Seifert fibered three manifold with base $B_0$; the 
characteristic class is an element of $H^2(B_0,\, {\mathbb Z})$. We can put a Sasakian structure 
on $S$ such that $B_0$ is the base of it with $p_1$ being the projection to the base of the 
Sasakian manifold. These are precisely the quasi-regular compact Sasakian three-manifolds.
Now let $Y$ be a complex projective manifold. Take $M\,=\, S\times Y$, $N\,=\, B_0\times Y$
and $p\,=\, p_1\times\text{Id}_Y$. It is easy to see that this they satisfy the two conditions
stated above.

Getting back the general situation,
the first of the two conditions means that $N$ is a smooth complex projective variety of complex 
dimension $d$. Note that the first condition implies that $D_N$ is a divisor of $N$. The 
second condition means that each irreducible component of $D_N$ is a smooth sub-variety of 
$N$ of dimension $d-1$, and these irreducible components intersect transversally. The 
first condition is a strong assumption. As mentioned before, $N$ is a normal variety, so 
smoothness is a strong assumption. The first condition rules out for example 
singularities isomorphic to the quotient of ${\mathbb C}^2$ by the involution $(x,\, y)\, 
\longmapsto\, (-x,\, -y)$. In view of the first condition, the second condition is rather 
mild.

\subsection{Holomorphic principal bundles}\label{se:holvec}

We now recall from \cite{BM} the definition of a Sasakian Higgs bundle.
The Riemannian metric and the Reeb vector field on $M$ will be denoted by $g$ and $\xi$
respectively.
The almost complex structure on the orthogonal complement $$\xi^\perp\, \subset\, TM$$
with respect to $g$ produces a type decomposition
$$
\xi^\perp\otimes_{\mathbb R} {\mathbb C}\,=\, F^{1,0}\oplus F^{0,1}\, .
$$
Let
\begin{equation}\label{w01}
\widetilde{F}^{0,1}\, :=\, F^{0,1}\oplus (\xi\otimes_{\mathbb R}
\mathbb C)\, \subset\, TM\otimes_{\mathbb R} {\mathbb C}
\end{equation}
be the distribution. It is known that
this distribution $\widetilde{F}^{0,1}$ is integrable \cite[p.~550, Lemma~3.2]{BS}.

Let $G$ be a complex reductive
affine algebraic group and $q\, :\, E_G\,\longrightarrow\, M$ a $C^\infty$ principal
$G$--bundle on $M$. Let $$dq\, :\, TE_G\,\longrightarrow\, q^* TM$$ be the differential
of the projection $q$. A partial 
connection on $E_G$ in the direction of $\xi$ is a $G$--equivariant homomorphism over $E_G$
$$
D_0\, :\, q^*\xi \,\longrightarrow\, TE_G
$$
such that $(dq)\circ D_0$ coincides with the identity map of the line sub-bundle $q^*\xi
\, \subset\, q^*TM$. In other words, 
$D_0$ is a $G$--equivariant lift of $\xi$ to the total space of $E_G$. A \textit{Sasakian 
principal $G$--bundle} on $M$ is a $C^\infty$ principal $G$--bundle $E_G$ on $M$ equipped 
with a partial connection $D_0$ in the direction of $\xi$.

Let $(E_G,\, D_0)$ be a Sasakian principal 
$G$--bundle on $M$ as above. A holomorphic structure on $E_G$ is a $C^\infty$ sub-bundle
\begin{equation}\label{hd}
{\mathcal D}\, \subset\, TE_G\otimes_{\mathbb R}{\mathbb C}
\end{equation}
such that the following four conditions hold:
\begin{enumerate}
\item the action of $G$ on $TE_G\otimes_{\mathbb R}{\mathbb C}$ given by the action of
$G$ on $E_G$ preserves ${\mathcal D}$,

\item the complexified differential
$$
dq\otimes_{\mathbb R}{\mathbb C}\, :\, TE_G\otimes_{\mathbb R}{\mathbb C}\,\longrightarrow\,
q^* (TM\otimes_{\mathbb R}{\mathbb C})\,=\, (q^*TM)\otimes_{\mathbb R}{\mathbb C}
$$
is an isomorphism from the sub-bundle ${\mathcal D}\, \subset\, TE_G\otimes_{\mathbb R}
{\mathbb C}$ to the sub-bundle $$q^*\widetilde{F}^{0,1} \, \subset\, q^* (TM\otimes_{\mathbb R}
{\mathbb C})\, ,$$ where $\widetilde{F}^{0,1}$ is constructed in \eqref{w01},

\item the sub-bundle ${\mathcal D}\, \subset\, TE_G\otimes_{\mathbb R}
{\mathbb C}$ is closed under the Lie bracket operation of vector fields of $E_G$, and

\item $D_0(q^*\xi)\otimes 1\, \subset\, {\mathcal D}$, meaning
the image of the homomorphism $D_0$ is contained in ${\mathcal D}$.
\end{enumerate}

A \textit{holomorphic Sasakian principal $G$--bundle} is
a Sasakian principal $G$--bundle equipped with a holomorphic structure. When $G\,=\,
\text{GL}(r, {\mathbb C})$, then a holomorphic Sasakian principal $G$--bundle is
a holomorphic Sasakian vector bundle of rank $r$ (see \cite{BS}).

Let $(E_G,\, D_0)$ be a Sasakian principal $G$--bundle. Note that any vector bundle 
associated to $E_G$ is a Sasakian vector bundle where the partial connection is given by 
$D_0$. In particular, the adjoint vector bundle $\text{ad}(E_G)$ is a Sasakian vector 
bundle. A holomorphic structure ${\mathcal D}$ on $(E_G,\, D_0)$ produces a holomorphic 
structure on any associated vector bundle (associated to $E_G$ by a holomorphic 
representation of $G$); see \cite{BM}. The trivial line bundle $M\times{\mathbb C}$
has a trivial holomorphic structure. The sub-bundle in \eqref{hd} for this
trivial holomorphic structure is the direct summand $q^*_0\widetilde{F}^{0,1}$
of $T(M\times{\mathbb C})\otimes {\mathbb C}\,=\, q^*_0\widetilde{F}^{0,1}\oplus
\widetilde{F}^{1,0}\oplus q_1^* T{\mathbb C}$, where $q_0$ and $q_1$ are the projections of
$M\times{\mathbb C}$ to $M$ and $\mathbb C$ respectively (the
holomorphic tangent bundle $T{\mathbb C}$ is of course trivial). A \text{holomorphic section} of a
holomorphic Sasakian vector bundle $((E,\, D_0),\,{\mathcal D})$ is a homomorphism from $M\times
{\mathbb C}$ to $E$ such that it takes the above direct summand $q^*_0\widetilde{F}^{0,1}$
to ${\mathcal D}$.

A Higgs field on a Sasakian holomorphic
principal $G$--bundle $(E_G,\, {\mathcal D})$ is a holomorphic section $\theta$ of
$\text{ad}(E_G)\otimes (F^{1,0})^*$ such that the following two conditions
hold:
\begin{enumerate}
\item the section $\theta\bigwedge\theta$ of
$\text{ad}(E)\otimes (\bigwedge^2 F^{1,0})^*$ vanishes identically, and

\item for any point $x\, \in\, N- D_N\,=\, N- p(D_M)$, where $p$ is the map in \eqref{p},
the restriction of the section $\theta$ to the loop $p^{-1}(x)$ is flat with respect
to the partial connection $D_0$ along $p^{-1}(x)$.
\end{enumerate}

The second condition means that the $G$--Higgs bundle $((E_G,\, D),\, \theta)$ is 
virtually basic (see \cite[Definition 4.2]{BM} for the definitions of basic and
virtually basic $G$--Higgs bundles).

\section{Ramified $G$--Higgs bundles}

\subsection{Ramified $G$--bundles}

Let $X$ be an irreducible smooth projective variety of dimension $d$ defined over 
$\mathbb C$. Let $D\, \subset\, X$ be a simple normal crossing divisor. As before, let 
$G$ be a complex reductive affine algebraic group. The Lie algebra of $G$ will be
denoted by $\mathfrak g$.

A \textit{ramified $G$--bundle} over $X$ with ramification over
the divisor $D$ is a smooth complex
quasi-projective variety $E_G$ equipped with a right
algebraic action of $G$
\begin{equation}\label{f}
f \, :\, E_G\times G\, \longrightarrow\, E_G
\end{equation}
and a surjective algebraic map
\begin{equation}\label{psi}
\psi\, :\, E_G\, \longrightarrow\, X\, ,
\end{equation}
such that the following five conditions hold:
\begin{itemize}
\item{} $\psi\circ f \, =\, \psi\circ p_1$, where $p_1$ is
the projection of $E_G\times G$ to $E_G$,

\item{} for each point $x\, \in\, X$, the action of $G$ on the
reduced fiber $\psi^{-1}(x)_{\text{red}}$ is transitive,

\item{} the restriction of $\psi$ to $\psi^{-1}(X- D)$ makes
$\psi^{-1}(X- D)$ a principal $G$--bundle over
$X- D$, meaning the map $\psi$ is smooth over
$\psi^{-1}(X- D)$ and the map to the fiber product
$$
\psi^{-1}(X- D)\times G\, \longrightarrow\,
\psi^{-1}(X- D)\times_{X- D} 
\psi^{-1}(X- D)
$$
defined by $(z,\, g) \longmapsto\, (z,\, f(z,\,g))$ is an
isomorphism,

\item{} for each irreducible component $D_i\, \subset\, D$,
the reduced inverse image $\psi^{-1}(D_i)_{\text{red}}
\, \subset\, \psi^{-1}(D_i)$ is a smooth divisor and
$$
\widehat{D}\, :=\, \sum_{i=1}^\ell \psi^{-1}(D_i)_{\text{red}}
$$
is a normal crossing divisor on $E_G$, and

\item{} for any smooth point $z\in \widehat{D}$, the isotropy
group $G_z\, \subset\, G$, for the action of $G$ on $E_G$,
is a finite cyclic group that acts faithfully on the
quotient line $T_zE_G/T_z\psi^{-1}(D)_{\text{red}}$.
\end{itemize}
(See \cite{BBN2}, \cite{Bi1}.)

Let $\text{Pvect}(X)$ denote the category of parabolic vector bundles over $X$ with 
parabolic structure over $D$ and rational parabolic weights. Let $\text{Rep}(G)$ denote 
the category of all finite dimensional rational left representations of $G$. A parabolic 
$G$--bundle over $X$ with $D$ as the parabolic divisor is defined to be a functor from 
$\text{Rep}(G)$ to $\text{Pvect}(X)$ that is compatible with the operations of taking 
direct sum, tensor product and dual. See \cite{BBN1}, \cite[Section 2]{Bi1}. This 
definition is based on \cite{No}.

There is a natural equivalence of categories
between parabolic $G$--bundles and ramified $G$--bundles.
(See \cite{BBN2}, \cite{Bi1}.)

\subsection{Ramified $G$--Higgs bundle}

Let $(E_G,\, \psi)$ be a ramified $G$--bundle
as in \eqref{psi}. The algebraic tangent bundle on
$E_G$ will be denoted by $TE_G$. Let
\begin{equation}\label{k0}
{\mathcal K} \, \subset\, TE_G
\end{equation}
be the sub-bundle defined by the orbits of the action of
$G$ on $E_G$. So for any $z\, \in\, E_G$, the fiber
${\mathcal K}_z\, \subset\, T_zE_G$ is the image of the
differential $$df_z(e)\, :\, {\mathfrak g}\, \longrightarrow\,
T_zE_G$$ of the map $f_z\, :\, G\, \longrightarrow\,
E_G$, $g\, \longmapsto\, f(z,\, g)$, where $f$ is the map
in \eqref{f}. Since $df_z(e)$ is an isomorphism onto its image,
which coincides with the vertical tangent space for $\psi$, we have
an algebraic isomorphism of vector bundles
\begin{equation}\label{eta}
\eta\,:\, E_G\times {\mathfrak g}\, \longrightarrow\,
{\mathcal K}\, .
\end{equation}
This $\eta$ is an isomorphism of sheaves of Lie algebras; the
Lie algebra operation on the sheaf of sections of $\mathcal K$
is given by the Lie bracket of vector fields.

Let ${\mathcal Q}$ denote the quotient
vector bundle $TE_G/{\mathcal K}$.
So we have a short exact sequence of vector bundles
\begin{equation}\label{ex.p.0}
0\, \longrightarrow\, {\mathcal K}\, \longrightarrow\, TE_G \,
\stackrel{q}{\longrightarrow}\,{\mathcal Q} \, \longrightarrow\,0
\end{equation}
over $E_G$. The action of $G$ on $E_G$ induces an action of
$G$ on the tangent bundle $TE_G$. This action of $G$ on
$TE_G$ clearly preserves the sub-bundle ${\mathcal K}$.
It may be mentioned that the isomorphism
$\eta$ in \eqref{eta} intertwines the action of $G$ on
${\mathcal K}$ and the diagonal action of $G$ constructed
using the adjoint action of $G$ on $\mathfrak g$.
Therefore, we have an induced action of $G$ on the quotient
bundle $\mathcal Q$.

Let
\begin{equation}\label{th0}
\theta_0\, \in\,
H^0(E_G,\, {\mathcal H}om({\mathcal Q}\, ,{\mathcal K}))
\, =\, H^0(E_G,\, {\mathcal K}\otimes{\mathcal Q}^*)
\end{equation}
be an algebraic section. We note that the actions of $G$
on ${\mathcal K}$ and ${\mathcal Q}$ together define
an action of $G$ on the complex vector space
$H^0(E_G,\, {\mathcal H}om({\mathcal Q}\, ,{\mathcal K}))$.

Combining
the exterior algebra structure of $\bigwedge
{\mathcal Q}^*$ and the Lie algebra structure on the fibers
of the vector bundle ${\mathcal K}\, =\, E_G\times{\mathfrak g}$
(see \eqref{eta}), we have a homomorphism
\begin{equation}\label{tau}
\tau\, :\, ({\mathcal K}\otimes{\mathcal Q}^*)\otimes
({\mathcal K}\otimes{\mathcal Q}^*)\, \longrightarrow\,
{\mathcal K}\otimes(\bigwedge\nolimits^2{\mathcal Q}^*)\, .
\end{equation}
So $\tau ((A_1\otimes\omega_1)\otimes(A_2\otimes
\omega_2))\,=\, [A_1\, ,A_2]\otimes (\omega_1\bigwedge
\omega_2)$. We will denote $\tau (a,\, b)$ also by $a\bigwedge b$.

\begin{definition}\label{def1}
A {\it Higgs field} on a ramified $G$--bundle
$E_G$ is a section
$$
\theta_0\, \in\,
H^0(E_G,\, {\mathcal K}\otimes {\mathcal Q}^*)
$$
{\rm as in \eqref{th0} satisfying the following two
conditions:}
\begin{enumerate}
\item the action of $G$ on $H^0(E_G,\, {\mathcal K}
\otimes{\mathcal Q}^*)$ leaves $\theta_0$ invariant, and
\item $\theta_0\bigwedge\theta_0\, =\, 0$ (see
\eqref{tau}).
\end{enumerate}
\end{definition}

\begin{definition}\label{def1p}
A {\it ramified Higgs $G$--bundle} is a
pair $(E_G\, ,\theta_0)$, where $E_G$ is a
ramified $G$--bundle, and $\theta_0$ is a
Higgs field on $E_G$.
\end{definition}

Let
\begin{equation}\label{dag}
{\mathcal A}_{E_G}\, :=\,
(\psi_*({\mathcal K}\otimes {\mathcal Q}^*))^G\,\subset\, \psi_*({\mathcal K}\otimes {\mathcal Q}^*)
\end{equation}
be the invariant direct image, where $\psi$ is the
projection in \eqref{psi}. Therefore,
\begin{equation}\label{e0}
H^0(X,\, {\mathcal A}_{E_G})\, =\,
H^0(E_G,\, {\mathcal K}\otimes{\mathcal Q}^*)^G\, .
\end{equation}

For $i\, \geq\, 0$, let
\begin{equation}\label{e00}
\widetilde{\mathcal K}_i\, :=\, (\psi_* ({\mathcal K}
\otimes (\bigwedge\nolimits^i{\mathcal Q}^*)))^G\,\subset\,
\psi_* ({\mathcal K}\otimes (\bigwedge\nolimits^i{\mathcal Q}^*))
\end{equation}
be the invariant direct image. So,
$\widetilde{\mathcal K}_1\, =\, {\mathcal A}_{E_G}$.
The homomorphism $\tau$ in \eqref{tau} yields a homomorphism
\begin{equation}\label{e1}
\widetilde{\tau}\, :\, \widetilde{\mathcal K}_1\otimes
\widetilde{\mathcal K}_1\, \longrightarrow\,
\widetilde{\mathcal K}_2\, .
\end{equation}

See \cite[Lemma 2.3]{Bi2} for the following lemma.

\begin{lemma}\label{lem1}
A Higgs field on $E_G$ is a section
$$
\theta\, \in\, H^0(X,\, {\mathcal A}_{E_G})
$$
such that
$$
\widetilde{\tau}(\theta,\, \theta)
\,=\, 0\, ,
$$
where $\widetilde{\tau}$ is constructed in \eqref{e1}.
\end{lemma}

\subsection{The adjoint vector bundle}

We noted earlier that there is a natural equivalence of categories
between parabolic $G$--bundles and ramified $G$--bundles
(see \cite{BBN2}, \cite{Bi1}). Let $E^P_G$ denote the
parabolic $G$--bundle corresponding to a ramified $G$--bundle
$E_G$. We also recall that $E^P_G$ associates a parabolic
vector bundle over $X$ to each object in $\text{Rep}(G)$. Let
\begin{equation}\label{add0}
\text{ad}(E_G)\, :=\, E^P_G (\mathfrak g)
\end{equation}
be the parabolic vector bundle over $X$ associated to the
parabolic $G$--bundle $E^P_G$ for the adjoint action of
$G$ on its Lie algebra $\mathfrak g$. This parabolic vector
bundle $\text{ad}(E_G)$ will be called the \textit{adjoint
vector bundle} of $E_G$. The vector bundle underlying the
parabolic vector bundle $\text{ad}(E_G)$ will
also be denoted by $\text{ad}(E_G)$. From the
context it will be clear which one is being referred to.

Consider the vector bundle ${\mathcal K}\, \longrightarrow\,
E_G$ constructed in \eqref{k0}. We noted that ${\mathcal K}$
is equipped with a natural action of $G$.
It is straight forward to check that the
invariant direct image $(\psi_*{\mathcal K})^G$
is identified with the vector bundle underlying the
parabolic vector bundle $\text{ad}(E_G)$ constructed in \eqref{add0}.
Indeed, this follows from the fact that for a usual principal
bundle, its adjoint vector bundle coincides with the invariant
direct image of the relative tangent bundle. Therefore, we have
\begin{equation}\label{add}
\text{ad}(E_G)\, =\, (\psi_*{\mathcal K})^G\,\subset\, \psi_*{\mathcal K}\, .
\end{equation}

There is a natural ${\mathcal O}_X$--linear homomorphism
\begin{equation}\label{a2}
\text{ad}(E_G)\otimes \Omega^1_X\, \longrightarrow\,
{\mathcal A}_{E_G}\, ,
\end{equation}
where ${\mathcal A}_{E_G}$ is constructed in \eqref{dag} \cite[(2.13)]{Bi2}.
This homomorphism is an isomorphism over the complement $X
- D$, but it is not an isomorphism over $X$ in general. In fact, ${\mathcal A}_{E_G}$
is the vector bundle underlying the parabolic tensor product of $\text{ad}(E_G)$ and
$\Omega^1_X$, so the usual tensor product $\text{ad}(E_G)\otimes \Omega^1_X$ is a sub-sheaf
of ${\mathcal A}_{E_G}$.

There is a natural homomorphism
$$
\tau'\, :\, \bigwedge\nolimits^2 (\text{ad}(E_G)\otimes \Omega^1_X)\,\longrightarrow
\, \text{ad}(E_G)\otimes \Omega^2_X\, .
$$
Consider the isomorphism in \eqref{a2}. This takes $\widetilde\tau$ in \eqref{e1} to the 
above homomorphism $\tau'$. Therefore, from Lemma \ref{lem1} it follows that a Higgs 
field on $E_G$ is a holomorphic section $\theta$ of $\text{ad}(E_G)\otimes \Omega^1_X$ such
that $\tau'(\theta\bigwedge\theta)\,=\, 0$.

Henceforth, a ramified bundle (respectively, ramified Higgs bundle) will also be
called a parabolic bundle (respectively, parabolic Higgs bundle).

\section{Parabolic Higgs bundles and Sasakian Higgs bundles}

In this section, we shall 
establish the equivalence between Sasakian $G-$Higgs bundles over Sasakian manifolds $M$ 
with base $N$ and parabolic $G-$Higgs bundles over the base $N$. Section \ref{p2s} will 
give us a way of going from a parabolic Higgs bundle to a Sasakian Higgs bundle while 
Section \ref{s2p} will give us the reverse path.

\subsection{From parabolic Higgs bundles to Sasakian Higgs bundles}\label{p2s}

Let $Y$ be a complex projective variety and $\Gamma$ a finite
group acting on $Y$ through algebraic
automorphisms. So we have a homomorphism
\begin{equation}\label{h}
h\, :\, \Gamma\, \longrightarrow\, \text{Aut}(Y)\, ,
\end{equation}
where $\text{Aut}(Y)$ is the group of all automorphisms of
the variety $Y$. A $\Gamma$--\textit{linearized} principal $G$--bundle over
$Y$ is a principal $G$--bundle
\begin{equation}\label{ph}
\phi\, :\, F_G\, \longrightarrow\, Y
\end{equation}
and an action of $\Gamma$ on the left of $F_G$
$$
\rho\, :\, \Gamma\times F_G\, \longrightarrow\, F_G
$$
such that the following two conditions hold:
\begin{itemize}
\item the actions of $\Gamma$ and $G$ on $F_G$ commute, and

\item $\phi\circ \rho (\gamma\, ,z)\, =\, h(\gamma)(\phi(z))$
for all $(\gamma\, ,z)\, \in\, \Gamma\times F_G$, where $h$
is the homomorphism in \eqref{h} and $\phi$ is the projection
in \eqref{ph}.
\end{itemize}

Let $\psi\, :\, E_G\, \longrightarrow\, N$ be a parabolic
principal Higgs $G$--bundle. Using the ``Covering lemma'' of
Kawamata (see \cite[Ch. 1.1, p. 303--305]{KMM}) it can
be shown that there is a finite Galois covering
\begin{equation}\label{vp}
\varphi\, :\, Y\, \longrightarrow\, N
\end{equation}
and a $\Gamma$--linearized principal $G$--bundle $F_G$
over $Y$, where $\Gamma\, :=\, \text{Gal}(\varphi)$ is
the Galois group, such that
\begin{equation}\label{v2}
E_G\, =\, \Gamma\backslash F_G
\end{equation}
\cite[Section 4.1]{Bi1}, \cite[Section 4.1]{Bi2}.

Now let $\theta$ be a Higgs field on the parabolic $G$--bundle $E_G$. There is a natural 
linear isomorphism between the Higgs fields on $E_G$ and the $\Gamma$--invariant Higgs 
fields on $F_G$ \cite[Proposition 4.1]{Bi2}. Let $\theta'$ be the $\Gamma$--invariant 
Higgs field on $F_G$ corresponding to the Higgs field $\theta$ on $E_G$.

Fix an ample holomorphic line bundle $L_0$ on $Y$. Define the tensor product
$$
L\, :=\, \bigotimes_{\gamma\in \Gamma} \gamma^*L_0\, ,
$$
which is an ample holomorphic line bundle on $Y$. Note that the action of $\Gamma$ on $Y$ 
has a natural lift to an action of $\Gamma$ on $L$. Take a Hermitian structure $h_0$ on 
$L_0$ such that the curvature $\text{Curv}(L_0,\, h_0)$ of $(L_0,\, h_0)$ is positive. 
Let
$$
h\, :=\, \bigotimes_{\gamma\in \Gamma} \gamma^*h_0
$$
be the Hermitian structure on $L$. The action of $\Gamma$ on $L$ clearly preserves $h$. 
Note that the curvature $\text{Curv}(L,\, h)$ of $(L,\, h)$ coincides with 
$\sum_{\gamma\in \Gamma}\gamma^*\text{Curv}(L_0,\, h_0)$, hence the $(1,\, 1)$--form 
$\text{Curv}(L,\, h)$ is positive.

Let
\begin{equation}\label{p2}
L \, \supset\, \{v\, \in\, L\, \mid\, h(v)\,=\, 1\}\,:=\, M_1
\, \stackrel{p}{\longrightarrow}\, Y
\end{equation}
be the principal ${\rm U}(1)$--bundle over $Y$. Using $h$, and the 
positive form $\text{Curv}(L,\, h)$ on $Y$, there is a regular Sasakian structure
on $M_1$ \cite{BG}. Since $h$ is preserved by the action of $\Gamma$ on $L$, the action
of $\Gamma$ on $L$ preserves $M_1$. The quotient $M\, :=\, M_1/\Gamma$ is a
quasi-regular Sasakian manifold with ${\rm U}(1)\backslash M\, =\, N$.

The pullback $(p^*F_G,\, p^*\theta')$ is a $G$--Higgs bundle
on the Sasakian manifold $M_1$, where $p$ is the projection in \eqref{p2} and $F_G$ is
the principal $G$--bundle in \eqref{v2}. The action of $\Gamma$ on
$(F_G,\, \theta')$ pulls back to an action of $\Gamma$ on $(p^*F_G,\, p^*\theta')$.
Consequently, $(p^*F_G,\, p^*\theta')$ produces a $G$--Higgs bundle on $M$.

Therefore, we have the following:

\begin{proposition}\label{prop1}
Given a ramified $G$--Higgs bundle $(E_G,\, \theta)$ on $N$, there is a Sasakian manifold
$M$ over $N$ and a $G$--Higgs bundle on $M$.
\end{proposition}

\subsection{From Sasakian Higgs bundles to parabolic Higgs bundles}\label{s2p}

Let $((E_G,\, {\mathcal D}),\, \theta)$ be a $G$--Higgs bundle on the Sasakian manifold
$M$ in Section \ref{se2.1}.

\begin{proposition}\label{prop2}
The quotient ${\rm U}(1)\backslash E_G$ is a ramified holomorphic principal $G$--bundle on
${\rm U}(1)\backslash M\,=\, N$.

The Higgs field $\theta$ produces a Higgs field on ${\rm U}(1)\backslash E_G$.
\end{proposition}

\begin{proof}
Consider the smooth complex projective variety $N$ and the simple normal crossing
divisor $D_N\,=\, p(D_M)$ on it. Let $\{D_i\}_{i=1}^n$ be the irreducible components
of $D_N$. For each irreducible component $D_i$, let $m_i\, >\, 0$ be the multiplicity of
$D_i$ associated to the projection $p$ from $M$. So $m_i$ is the order of the isotropy subgroup
of a general point of $p^{-1}(D_i)$ for the action of ${\rm U}(1)$ on $M$ (this
uses Assumption \ref{assum}). Given
this collection of pairs $\{(D_i,\, m_i)\}_{i=1}^n$ the ``covering lemma'' of Kawamata says the
following:

There is a smooth projective variety $Z$ and a (ramified) Galois covering
\begin{equation}\label{beta}
\beta\, :\, Z\, \longrightarrow\, N
\end{equation}
such that
\begin{enumerate}
\item the reduced divisor $\beta^{-1}(D_N)_{\rm red}$ is a simple normal crossing divisor
on $Z$, and

\item $\beta^{-1}(D_i)\,=\, k_im_i\beta^{-1}(D_i)_{\rm red}$ for all $1\, \leq\, i\, \leq\,
n$, where $k_i$ are positive integers.
\end{enumerate}
(See \cite[Theorem 1.1.1]{KMM}, \cite[Theorem 17]{Ka}.) The Galois group $\text{Gal}(\beta)$
for the covering in \eqref{beta} will be denoted by $\Gamma$.

Let $$\beta_1\, :=\, \beta\vert_{\beta^{-1}(N- D_N)}\, :\, \beta^{-1}(N- D_N)
\,\longrightarrow\, N- D_N $$ be the restriction of $\beta$. Consider the
principal ${\rm U}(1)$--bundle $M'\,:=\, M- D_M\, \longrightarrow\, N- D_N$ in
\eqref{p}. The pulled back principal ${\rm U}(1)$--bundle
$$
\beta^*_1 M'\, \longrightarrow\, \beta^{-1}(N- D_N)
$$
extends to a principal ${\rm U}(1)$--bundle $\widehat{Z}$
on $Z$. Indeed, this follows from the fact that
$\beta^{-1}(D_i)\,=\, k_im_i\beta^{-1}(D_i)_{\rm red}$. This $\widehat{Z}$ is a regular
Sasakian manifold with ${\rm U}(1)\backslash \widehat{Z}\,=\, Z$, and it fits in a commutative
diagram
$$
\begin{matrix}
\widehat{Z} &\stackrel{\delta}{\longrightarrow} & M\\
~\Big\downarrow \mu && ~\Big\downarrow p\\
Z &\stackrel{\beta}{\longrightarrow} & N
\end{matrix}
$$
where $p$ and $\beta$ are in the maps in \eqref{p} and \eqref{beta} respectively. The
map $\delta$ is a Galois covering with Galois group $\Gamma\, =\,\text{Gal}(\beta)$. 

Consider the pulled back $G$--Higgs bundle $((\delta^*E_G,\, \delta^*{\mathcal D}),\, 
\delta^*\theta)$ on $\widehat{Z}$. Since $\delta$ is a Galois covering with Galois group 
$\Gamma$. The Galois group $\Gamma$ acts on $((\delta^*E_G,\, \delta^*{\mathcal D}),\, 
\delta^*\theta)$. The actions of ${\rm U}(1)$ and $\Gamma$ on $\delta^*E_G$ commute. 
Since $\widehat{Z}$ is a regular Sasakian manifold, and ${\rm U}(1)\backslash 
\widehat{Z}\,=\, Z$, we conclude that there is a $G$--Higgs bundle $(E'_G,\, \theta')$ on 
$Z$ such that $((\delta^*E_G,\, \delta^*{\mathcal D}),\, \delta^*\theta)$ is the pullback 
of $(E'_G,\, \theta')$ to $\widehat{Z}$. The action of $\Gamma$ on $((\delta^*E_G,\, 
\delta^*{\mathcal D}),\, \delta^*\theta)$ produces an action of $\Gamma$ on the 
$G$--Higgs bundle $(E'_G,\, \theta')$. Hence $(E'_G,\, \theta')$ produces a ramified 
$G$--Higgs bundle on $Z/\Gamma\,=\, N$ (see \cite[Proposition 4.1]{Bi2}). Let $(E''_G,\, 
\theta'')$ be the ramified $G$--Higgs bundle on $N$ defined by $(E'_G,\, \theta')$.

Now it is straight-forward to check that $E''_G$ coincides with the 
quotient ${\rm U}(1)\backslash E_G$. Furthermore, the Higgs field $\theta''$ coincides 
with $\theta$ on $p^{-1}(N- D_N)$.
\end{proof}

The construction in Section \ref{p2s} of a Sasakian $G$--Higgs bundle from a parabolic $G$--Higgs 
bundle and the construction in Section \ref{s2p} of a parabolic $G$--Higgs bundle from a Sasakian 
$G$--Higgs bundle are evidently inverses of each other.

\section*{Acknowledgements}

We thank the referees for their helpful comments.
The authors acknowledge the support of their respective J. C. Bose Fellowships.

\end{document}